\documentclass[12pt]{article}

\usepackage[a4paper, total={7in, 9in}]{geometry}

\usepackage[english]{babel}
\usepackage[utf8]{inputenc}
\usepackage{amsmath}
\usepackage{graphicx}
\usepackage[colorinlistoftodos]{todonotes}
\usepackage{amssymb}
\usepackage{amsthm}
\usepackage{multicol}
\usepackage{hyperref}
\usepackage{fancyhdr}

\newtheorem{dfn}{Definition}[section]
\newtheorem{thm}{Theorem}[section]
\newtheorem{prop}{Proposition}[section]
\newtheorem{lem}{Lemma}[section]

\newtheorem{cor}{Corollary}[section]
\newtheorem{obs}{Observation}[section]

\DeclareMathOperator{\fr}{FR}

\title{Are Ramsey Algebras Essentially Semigroups\footnote{The results of this paper were presented at the 2016 Asian Mathematical Conference in Bali, Indonesia.}}

\author{%
\textsc{Zu Yao Teoh} \\[1ex] % Your name
\normalsize School of Mathematical Sciences, \\ Universiti Sains Malaysia, \\ 11800 USM, Malaysia. \\ % Your institution
\normalsize \href{mailto:teohzuyao@gmail.com}{teohzuyao@gmail.com} % Your email address
\and % Uncomment if 2 authors are required, duplicate these 4 lines if more
\textsc{Andrew Rajah} \\[1ex] % Second author's name
\normalsize School of Mathematical Sciences, \\ Universiti Sains Malaysia, \\ 11800 USM, Malaysia. \\ % Second author's institution
\normalsize \href{mailto:andy@usm.my}{andy@usm.my} % Second author's email address
\and
\textsc{Wen Chean Teh}\thanks{Corresponding author.} \\[1ex] % Second author's name
\normalsize School of Mathematical Sciences, \\ Universiti Sains Malaysia, \\ 11800 USM, Malaysia. \\ % Second author's institution
\normalsize \href{mailto:dasmenteh@usm.my}{dasmenteh@usm.my} % Second author's email address
}

\date{}

\begin{document}
\maketitle

\abstract{It is known that semigroups are Ramsey algebras. This paper is an attempt to understand the role associativity plays in a binary system being a Ramsey algebra. Specifically, we show that the nonassociative Moufang loop of octonions is not a Ramsey algebra.} \\

{\bf{2010 Mathematics Subject Classification:}} 03E05, 05D10

\emph{Keywords}: Binary systems, Ramsey algebras, octonions, semigroups, associativity, nonassociative Moufang loops.

\section{Introduction}
Ramsey algebra has its roots in Hindman's theorem \cite{NH74}. Hindman's theorem states that, for each finite partition of the positive integers $\mathbb{Z}^+$, there exists an infinite $H\subseteq\mathbb{Z}^+$ contained in one of the pieces $A$ such that $a_1+\cdots+a_n\in A$ whenever $a_1, \ldots, a_n$ are distinct elements of $H$. (One can also formulate Hindman's theorem in terms of finite sets of natural numbers and the set theoretic operation $\cup$. See Milliken \cite{milliken1975ramsey} for details.) This result can be viewed as a combinatorial result on the algebra $(\mathbb{Z}^+, +)$. Ramsey algebra, a name suggested by Carlson, prescribes Ramsey type combinatorics to algebras. Early works on the topic by Teh can be found in \cite{teh2014ramsey}, \cite{wcT13a}, \cite{wcT13b}, and \cite{tehidem16}. The notion of Ramsey algebra came into conception when Carlson singled out the class of Ramsey spaces induced by algebras and saw the potential of a purely combinatorial study of such spaces. Ramsey spaces had previously been introduced by Carlson \cite{carlson1988some} as a generalization to the Ellentuck space introduced in conjunction with Ellentuck's theorem \cite{ellentuck1974new}, which generalizes the results by Galvin \& Prikry \cite{GP73} and Silver \cite{silver1970every}.

It is known that all semigroups are Ramsey algebras. It is also known that $(\mathbb{Z}, -)$, where $-:\mathbb{Z}^2\rightarrow\mathbb{Z}$ is defined by $-(x, y)=y-x$, is not a Ramsey algebra (Theorem 2.3.2, page 41, \cite{tehthesis}). Effort to identify more binary systems that are Ramsey has, except for some pedagogical ones, been met with failure thus far. $(\mathbb{Z}, -)$ is a good example of an algebraic system not too remote from a semigroup that turns out not to be a Ramsey algebra. We are thus led to ask if associativity is the salient property that makes all semigroups Ramsey algebras. In this paper, we strive to investigate further at the role associativity plays in determining whether a binary system is Ramsey. In particular, we will show that the octonions with multiplication forming a nonassociative Moufang loop is not a Ramsey algebra.

\section{Preliminaries}
The set of natural numbers $0, 1, 2, \ldots$ will be denoted by $\omega$; the set of positive integers, i.e. $\omega\setminus\{0\}$, is denoted by $\mathbb{Z}^+$. Infinite sequences will be emphasized by an arrow over a letter such as $\vec{b}$ and $\vec{a}$.

\subsection{The Octonions with Mutiplication}\label{basicoctfacts}
The (real) octonions are real linear combinations of the eight unit octonions $e_0, e_1, \ldots, e_7$. Multiplication of octonions is nonassociative and noncommutative. The product $a\cdot b$ of octonions $a=\sum_{i=0}^7 a_ie_i, b=\sum_{j=0}^7 b_je_j$ is given by $\sum_{i, j\in\{0, \ldots, 7\}}a_ib_je_ie_j$, where the products of the unit octonions are given in Table \ref{octmult}.

For the product of three or more octonions, nonassociativity begins to factor in. If $a=\sum_{i=0}^7 a_ie_i, b=\sum_{j=0}^7 b_je_j$, and $c=\sum_{k=0}^7c_ke_k$ are octonions, then the products $(ab)c$ and $a(bc)$ are distinct in general. It is instructive to note that the difference between these two products has root in the product of unit octonions under the two different bracketings. The product $(ab)c$ is given by
\begin{equation}
(ab)c=\sum_{i, j, k\in\{0 \ldots, 7\}}a_ib_jc_k(e_ie_j)e_k,
\end{equation}
whereas the product $a(bc)$ is given by
\begin{equation}
a(bc)=\sum_{i, j, k\in\{0 \ldots, 7\}}a_ib_jc_ke_i(e_je_k).
\end{equation}
Thus, we see that the \emph{coefficient} $a_ib_jc_k$ of each corresponding summand in the two products are equal for the same string of unit octonions involving $e_i, e_j, e_k$ in the order indicated; we conclude that, if associativity fails in this product, then it is the bracketings of the unit octonions that play a role. This observation applies to the products of three or more octonions in general; Corollary \ref{uptosign} and Proposition \ref{mainprop} will address this aspect in greater detail. Also related are the sets $\Lambda_j^t$ to be defined in Section 4. Other pertinent properties of products of the unit octonions will be given in the next section.

\begin{table}\label{octmult}
\centering
    \begin{tabular}{ | c || c | c | c | c | c | c | c | c |}
    \hline
     $\cdot$ & $e_0$ & $e_1$ & $e_2$ & $e_3$ & $e_4$ & $e_5$ & $e_6$ & $e_7$ \\ \hline
    $e_0$ & $e_0$ & $e_1$ & $e_2$ & $e_3$ & $e_4$ & $e_5$ & $e_6$ & $e_7$  \\ \hline
    $e_1$ & $e_1$ & $-e_0$ & $e_3$ & $-e_2$ & $e_5$ & $-e_4$ & $-e_7$ & $e_6$ \\ \hline
    $e_2$ & $e_2$ & $-e_3$ & $-e_0$ & $e_1$ & $e_6$ & $e_7$ & $-e_4$ & $-e_5$ \\ \hline
    $e_3$ & $e_3$ & $e_2$ & $-e_1$ & $-e_0$ & $e_7$ & $-e_6$ & $e_5$ & $-e_4$ \\ \hline
    $e_4$ & $e_4$ & $-e_5$ & $-e_6$ & $-e_7$ & $-e_0$ & $e_1$ & $e_2$ & $e_3$ \\ \hline
    $e_5$ & $e_5$ & $e_4$ & $-e_7$ & $e_6$ & $-e_1$ & $-e_0$ & $-e_3$ & $e_2$ \\ \hline
    $e_6$ & $e_6$ & $e_7$ & $e_4$ & $-e_5$ & $-e_2$ & $e_3$ & $-e_0$ & $-e_1$ \\ \hline
    $e_7$ & $e_7$ & $-e_6$ & $e_5$ & $e_4$ &$-e_3$ & $-e_2$ & $e_1$ & $-e_0$ \\ \hline
    \end{tabular}
\caption{Multiplication table for the unit octonions.}
\end{table}

\subsection{Binary Systems, Bracketed Strings, Assignments}
Our notion of a Ramsey algebra will be made in the setting of binary systems; a binary system $(G, \cdot)$ is an algebra where $\cdot$ is a binary operation on $G$. The multiplicative notation will be assumed throughout this paper; specifically, the symbol $\cdot$ will be omitted when the context is clear. Towards this end, let $G$ denote a nonempty set and $\cdot$ a binary operation on $G$. We also fix a set $V=\{x_0, x_1, \ldots\}$ of variables throughout.

Strings of $V$ will eventually be given values in $G$. To be more precise, it is strings with bracketing that will be given values in $G$. We first have to define the notion of a bracketing on strings or \emph{bracketed} strings; this will be done on an arbitrary set $A$. Bracketed strings are defined recursively as follows:
\begin{dfn}
The set $\mathcal{T}_A$ of bracketed strings of $A$ is the set of strings of $A\cup\{(, )\}$ such that
\begin{enumerate}
\item each $a\in A$ is a bracketed string, and
\item if $t_1, t_2$ are bracketed strings, then $(t_1t_2)$ is a bracketed string.
\end{enumerate}
\end{dfn}

Readers familiar with logical terminologies will notice that bracketed strings of $A$ are similar to what are known as \emph{terms}, hence the notation $\mathcal{T}$ to denote the set of bracketed strings. If $t$ is a bracketed string, we write $\check{t}$ to denote the string obtained from $t$ by omitting the brackets and we call $\check{t}$ \emph{unbracketed}.

A function $\mu:V\rightarrow G$ will be known as an \emph{assignment} on $V$ to abstract the idea of $\mu$ assigning a value to each variable in $V$, value of which lying in $G$.

\begin{dfn}
For each $t\in\mathcal{T}_V$ and each assignment $\mu:V\rightarrow G$, we will define $t^\mu\in G$ recursively as follows:
\begin{enumerate}
\item For each $n\in\omega$, ${x_n}^\mu=\mu(x_n)$.
\item If $(t_1t_2)\in\mathcal{T}_V$, then $(t_1t_2)^\mu={t_1}^\mu{t_2}^\mu$.
\end{enumerate}
\end{dfn}

We remark that, for any $G$, each $t\in\mathcal{T}_G$ has its natural interpretation. For instance, if $G$ is the set of unit octonions, then $t=(e_1(e_2e_3))$ \emph{evaluates} to $-e_0$. On the other hand, given any string in $s=\mathcal{T}_G$, there exists an assignment $\mu$ and $t\in\mathcal{T}_V$ such that $t^\mu=s$ (equality as elements of $G$).

\subsection{Reductions and Ramsey Algebras}
The concept of a reduction is central to the notion of a Ramsey algebra. To define the concept, let us identify each unbracketed string $\check{t}=a_0\cdots a_n$ of $A$ with its corresponding finite sequence $\langle a_0, \ldots, a_n\rangle$ and when $t, t'$ are finite sequences, we write $t^\frown t'$ to denote the concatenation of $t$ with $t'$.
\begin{dfn}\label{reduction}
If $\vec{a}=\langle a_0, a_1, \ldots\rangle$ and $\vec{b}=\langle b_0, b_1, \ldots\rangle$ are infinite sequences of $G$  such that there exist bracketed strings $t_i$ of $G$ for each $i\in\omega$ such that
\begin{enumerate}
\item $\check{t}_0^\frown\check{t}_1^\frown\cdots$ forms an infinite subsequence of $\vec{b}$ and
\item $t_i$ evaluates to $a_i$ for each $i\in\omega$,
\end{enumerate}
then we say that $\vec{a}$ is a \emph{reduction} of $\vec{b}$, denoted by $\vec{a}\leq\vec{b}$.
\end{dfn}
If $\vec{b}$ is an infinite sequence of $G$, then any infinite subsequence of $\vec{b}$ is a trivial example of a reduction of $\vec{b}$. For another example, if $\vec{a}, \vec{b}$ are such that $a_0=b_1$, $a_1$ is the value of $(b_2(b_3b_4))$, $a_2$ is the value of $((b_6b_7)b_9)$, $\ldots$, then $\vec{a}$ is a reduction of $\vec{b}$. We mention in passing that $\leq$ is \emph{reflexive} and \emph{transitive} on the set of infinite sequences of $G$.

One subclass of $\mathcal{T}_V$ is of particular interest, namely the class of $t\in\mathcal{T}_V$ such that the indices of the variables occurring in $t$ are strictly increasing from left to right; members of this class will be called \emph{orderly} bracketed strings of variables. We denote this class by $\mathcal{OT}_V$. This class of strings is important when we are concerned with reductions. Specifically, the fact that $\check{t}_0^\frown\check{t}_1^\frown\cdots$ in Definition \ref{reduction} forming a subsequence of $\vec{b}$ forces each of the $t_i$ to be ``orderly." Furthermore, if $t, t'\in\mathcal{OT}_V$, we write $t\prec t'$ to mean that the greatest index occurring in the variables of $t$ is less than the least index in the variables occurring in $t'$. Thus, $\vec{a}\leq\vec{b}$ if and only if there exist orderly bracketed strings of variables $t_i\prec t_{i+1}$ for each $i\in\omega$ such that $a_i=t_i{^\mu}$ under the assignment $\mu(x_n)=b_n$, $n\in\omega$.

The sets $\fr(\vec{a})$'s are also integral to the notion of a Ramsey algebra:
\begin{dfn}
Given an infinite sequence $\vec{a}$ of $G$, define $g\in\fr(\vec{a})\subseteq G$ if and only if $g$ is the value of a bracketed string $t$ of terms of $\vec{a}$ such that $\check{t}$ is a finite subsequence of $\vec{a}$.
\end{dfn}

In the orderly terminology, $g\in\fr(\vec{a})$ if and only if $g=t^\mu$ for some orderly $t\in\mathcal{T}_V$, where $\mu$ is the assignment replacing $x_i$ with the $i$th term of $\vec{a}$. Note that, for any sequence $\vec{a}$ of $G$, the set $\fr(\vec{a})$ is nonempty since each term of $\vec{a}$ is an element of the set.

We can now define when a binary system $(G, \cdot)$ is a Ramsey algebra.

\begin{dfn}[Ramsey algebra]
Let $(G, \cdot)$ be a binary system. Then $(G, \cdot)$ is said to be a Ramsey algebra if, given an $X\subseteq G$ and an infinite sequence $\vec{b}$ of elements of $G$, there exists an $\vec{a}\leq\vec{b}$ such that $\fr(\vec{a})\subseteq X$ or $\fr(\vec{a})\subseteq X^C$.
\end{dfn}

We want to emphasize that the definitions above are adapted to deal with the class of binary systems. The general treatment of Ramsey algebras can be found in \cite{teh2014ramsey} and a generalization to heterogeneous algebras can be found in \cite{teohteh}.

\subsection{Finite Moufang Loops and $M(G, 2)$}
We take a brief digression to see that all finite Moufang loops as well as the class of Moufang loops $M(G, 2)$, one for each group $G$, are Ramsey algebras. In either of these cases, the key to being a Ramsey algebra owes to the fact that every given infinite sequence has a ``nice" reduction that leads to the binary system being Ramsey.

\begin{thm}
Every finite Moufang loop is a Ramsey algebra.
\end{thm}
\begin{proof}
Given any infinite sequence of the Moufang loop, one of the elements of the loop will occur infinitely often in the sequence. Since each element of the loop has finite order, we can find a reduction of the given sequence made up of the identity element.
\end{proof}

Now, let $(G, \cdot)$ be a group and introduce a new symbol $u$ not already in $G$. Define $Gu$ to be the set of symbols $\{gu:g\in G\}$. Then, $M=G\cup Gu$, along with the binary operation $\ast$ to be introduced below, is a Moufang loop by adopting the following rules for $\ast$:
\begin{enumerate}
\item $\ast$ restricted to $G$ coincides with $\cdot$.
\item For all $g, h\in G$,
\begin{enumerate}
\item $(gu)h=(gh^{-1})u$,
\item $g(hu)=(hg)u$,
\item $(gu)(hu)=h^{-1}g$.
\end{enumerate}
\end{enumerate}
This Moufang loop is denoted by $M(G, 2)$ and it is associative if and only if $G$ is abelian.

Rule 2(c) in the list is the key to the following theorem:
\begin{thm}
$M(G, 2)$ is a Ramsey algebra for any group $G$.
\end{thm}
\begin{proof}
The idea is that every infinite sequence of $M$ has a reduction consisting of elements of $G$: Given an infinite sequence of $M$, we can find a reduction of the sequence consisting only of group elements, either by taking an infinite subsequence or applying Rule 2(c) in the list above. Then, since groups are Ramsey algebras, we have the desired result.
\end{proof}

\section{Preparatory Results}
This section will set up the required components for Section 4.
\subsection{Basic Properties of Products of Unit Octonions}
Some basic properties of octonion multiplication will come in handy. A glance through the multiplication table reveals two immediate properties:
\begin{equation}
e_ie_j=-e_je_i\;\text{for all distinct, nonzero}\;i, j.
\end{equation}
and
\begin{equation}
e_i^2=-e_0\;\text{for all}\;i\neq 0.
\end{equation}
Of course, as the identity element, $e_0^2=e_0$.

Less evident is the fact that, for distinct, nonzero $i, j, k$ such that $e_ie_j\neq e_k$, we have
\begin{equation}
e_i(e_je_k)=-(e_ie_j)e_k.
\end{equation}
In addition, because the unit octonions form a Moufang loop, we have what are called the left, right alternative identities and the flexible identity, and they yield $e_i(e_ie_j)=(e_i^2)e_j$, $(e_ie_j)e_j=e_i(e_j^2)$, and $e_i(e_je_i)=(e_ie_j)e_i$.

Thus, the product of three unit octonions under different bracketings are equal up to a sign difference. More of this will be discussed in the next subsection.

\subsection{Strings of Unit Octonions}
Strings of unit octonions play an important role in our analysis of products of octonions. We begin by singling out a class of bracketed strings called right-associative strings of octonions, where all brackets are associated to the right. For example, $(e_0(e_1(e_2e_3)))$ is a right-associative bracketing of the string $e_0e_1e_2e_3$. This class of strings will serve as a point of reference to compare with other bracketed strings.

Call two octonions $e, f$ equal \emph{up to a sign difference} if $f$ and $e$ are equal or each is the negative of the other and denote the relation by $f\sim e$. (Note that equality up to sign is an equivalence relation on the octonions.)

\begin{lem}
Every bracketed string of unit octonions evaluates, up to a sign difference, to the value of its right-associative form.
\end{lem}
\begin{proof}
The base case is trivial as the bracketings are each unique.

Now, given the bracketed string $t$ of unit octonions such that  $\check{t}=e_{i_1}\cdots e_{i_{N+1}}$, let $t$ be such that $t=(t_1t_2)$, where $t_1, t_2\in\mathcal{T}_{\{e_0, \ldots, e_7\}}$ and $\check{t}_1=e_{i_1}\cdots e_{i_n}$ for some $n\in\{1, \ldots, N\}$.

Let $s_1, s_2$ be the right-associative forms corresponding respectively to $t_1, t_2$; we also let $s_1', s\in\mathcal{T}_{\{e_0, \ldots, e_7\}}$ be such that $s_1=(e_{i_1}s_1')$ and $s$ the right-associative form of $(s_1's_2)$. Then, applying induction hypothesis and the basic properties of octonion multiplication as required, we have
\begin{eqnarray}
(t_1t_2) &\sim& (s_1s_2) \\
&\sim& ((e_{i_1}s_1')s_2) \\
&\sim& (e_{i_1}(s_1's_2)) \\
&\sim& (e_{i_1}s), \label{raform}
\end{eqnarray}
where we note that (\ref{raform}) is in the right-associative form.
\end{proof}

\begin{cor}\label{uptosign}
Up to a sign difference, different bracketings of a string of unit octonions evaluate to the same unit octonion.
\end{cor}
\begin{proof}
By the transitivity of $\sim$.
\end{proof}

\begin{prop}\label{mainprop}
Suppose $t, t'\in\mathcal{OT}_V$ are distinct and are such that $\check{t}=\check{t}'$. Then, there exists an assignment $\mu:V\rightarrow\{e_0, \ldots, e_7\}$ such that $t^\mu=e_4=-t'^\mu$.
\end{prop}
\begin{proof}
We prove by induction on the complexity of bracketed strings in $\mathcal{OT}_V$.

Let $t, t'\in\mathcal{OT}_V$ be such that $\check{t}=\check{t}'$. If $t, t'$ each has $3$ variables, then one is $(v_1(v_2v_3))$ and the other is $((v_1v_2)v_3)$ for some $v_1, v_2, v_3\in V$. Use any assignment $\mu$ where $v_1, v_2, v_3$ are respectively assigned $e_5, e_6, e_7$. Then, the former is evaluated to $-e_4$ while the latter is evaluated to $e_4$.

For the inductive step, let $t=(t_1t_2)$ and $t'=(t_1't_2')$ for some nonempty orderly bracketed strings $t_1\prec t_2, t_1'\prec t_2'$ such that $\check{t}=v_1\cdots v_N=\check{t}'$. Then we have $t_1\neq t_1'$ or $t_2\neq t_2'$.

If $t_1\neq t_1'$, two possibilities can arise, namely $\check{t}_1=\check{t}_1'$ or $\check{t}_1\neq\check{t}_1'$. In the former case, we may apply induction hypothesis to obtain an assignment $\mu$ such that $t_1^\mu=-t_1'{^\mu}$ while the variables in $t_2$ and $t_2'$ are all assigned $e_0$; then, $t^\mu=-t'^{\mu}$. In the latter case, let $\check{t}_1$ be a subsequence of $\check{t}_1'$ without loss of generality and let $k$ be such that $1<k<N$ and $v_k$ is a symbol in $t_1'$, but not in $t_1$. Then, assign $e_5$ to $v_1$, $e_6$ to $v_k$, and $e_7$ to $v_N$ while all other symbols are assigned $e_0$. One computes $t$ to evaluate to $e_4$ while $t'$ evaluates to $-e_4$.

If $t_1=t_1'$, then it must be the case that $t_2\neq t_2'$. Similar argument as the case above then applies. This completes the proof.
\end{proof}

Note that we have arrived at $e_4$ out of wit; we could have chosen assignments that evaluate to other unit octonions in place of $e_4$.

\section{The Main Result}
We will show that the octonions with multiplication is not a Ramsey algebra by exhibiting a bad sequence $\vec{b}$ and a set $X\subseteq\mathbb{O}$ such that for each $\vec{a}\leq\vec{b}$, we have $\fr(\vec{a})\cap X\neq\varnothing$ and $\fr(\vec{a})\cap X^C\neq\varnothing$, the contrapositive of the defining statement of a Ramsey algebra. The bad sequence $\vec{b}=\langle b_0, b_1, \ldots\rangle$ is given by:
\begin{equation}
b_n=2^{2^{8n+1}}e_0+\cdots+2^{2^{8n+8}}e_7=\sum_{i=0}^7 2^{2^{8n+1+i}}e_i. \label{badseq}
\end{equation}
This bad sequence will be called $\vec{b}$ throughout this section.

We will be appealing to the nonadjacent form of representation of the integers in the proof of our main theorem. The $n$-digit nonadjacent form representation (NAF) of an integer $a$ is $q_{n-1}q_{n-2}\cdots q_0$ with
\begin{equation}
a=\sum_{j=0}^{n-1}q_j2^j,
\end{equation}
where $q_j\in\{-1, 0, 1\}$, $j\in\{0, \ldots, n-1\}$, and $q_j\times q_{j+1}=0$ for each $j\in\{0, \ldots, n-2\}$. This representation of the integers is \emph{unique} \cite{prep71}. NAF, as its name suggests, is a representation of integers such that there is at least a $0$ between any adjacent pair of nonzero digits. We will capitalize on the uniqueness of the NAF representation throughout this section.

In the proof of the main theorem, we will encounter quantities of the form $2^{8n_1+1+\alpha_1}+\cdots+2^{8n_N+1+\alpha_N}$ in abundance, where $n_1<\cdots<n_N$ are nonnegative integers and $\alpha_i\in\{0, \ldots, 7\}$ for $i\in\{1, \ldots, N\}$. Note that these quantities are essentially the usual binary representations of the integers. Hence, for different sets of $n_1<\cdots<n_N$ or $\alpha_i\in\{0, \ldots, 7\}$, the quantities so defined are \emph{distinct}. In fact, since each of these quantities are multiples of $2$, the difference between any distinct two is \emph{at least $2$}. This observation is of particular importance for an application of the uniqueness of NAF in the proof of our main theorem:

\begin{obs}\label{mainobs}
For different sets of $n_1<\cdots<n_N$ or $\alpha_i\in\{0, \ldots, 7\}$, the quantities of the form $2^{8n_1+1+\alpha_1}+\cdots+2^{8n_N+1+\alpha_N}$ are \emph{distinct} and are with difference of at least $2$.
\end{obs}

Before proceeding to the main theorem, we introduce some notations and conventions to be used in the proof. For each $t\in\mathcal{OT}_V$ and $\alpha\in\{0, \ldots, 7\}^N$ such that $\check{t}=x_{n_1}\cdots x_{n_N}$, we let $\mu_\alpha^t:V\rightarrow\{e_0, \ldots, e_7\}$ be any assignment such that $\mu_\alpha^t(x_{n_i})=e_{\alpha_i}$ for each $i\in\{1, \ldots, N\}$. For each such $t$, denote by $\Lambda_j^t$ the set of $\alpha=(\alpha_1, \ldots, \alpha_N)\in\{0, \ldots, 7\}^N$ such that $t^{\mu_\alpha^t}$ evaluates to $e_j$ up to sign difference. Note that if $t\neq t'$ are such that $\check{t}=\check{t}'$, then Corollary \ref{uptosign} implies that $\Lambda_j^t=\Lambda_j^{t'}$. 

\begin{thm}
The octonions with multiplication is not a Ramsey algebra.
\end{thm}
\begin{proof}
Throughout the proof, let $\mu$ denote the assignment $\mu(x_n)=b_n$. We will be using the bad sequence given by (\ref{badseq}) and the set $X\subseteq\mathbb{O}$ given by
\begin{equation}
X=\left\{(t_1(t_2t_3))^\mu:t_1, t_2, t_3\in\mathcal{OT}_V\;\text{and}\;t_1\prec t_2\prec t_3\right\}.
\end{equation}

If $t\in\mathcal{OT}_V$ such that $\check{t}=x_{n_1}\cdots x_{n_N}$, observe that
\begin{eqnarray}
t^\mu &=& \sum_{\alpha=(\alpha_1, \ldots, \alpha_N)\in\{0, \ldots, 7\}^N}2^{2^{8n_1+1+\alpha_1}+\cdots+2^{8n_N+1+\alpha_N}}t^{\mu_\alpha^t} \\
&=& \sum_{j=0}^7\left(\sum_{\alpha\in\Lambda_j^t}p_\alpha^t\right)e_j, \label{prod}
\end{eqnarray}
where, for each $\alpha\in\Lambda_j^t$, $p_\alpha^t$ is defined by
\begin{displaymath}
   p_\alpha^t = \left\{
     \begin{array}{lr}
       2^{2^{8n_1+1+\alpha_1}+\cdots+2^{8n_N+1+\alpha_N}} & \text{if}\; t^{\mu_\alpha^t}=e_j \\
       -2^{2^{8n_1+1+\alpha_1}+\cdots+2^{8n_N+1+\alpha_N}} & \text{if}\; t^{\mu_\alpha^t}=-e_j.
     \end{array}
   \right.
\end{displaymath}

{\bf{Claim}}: Suppose that $t_1\prec t_2\prec t_3$ and $t_1'\prec t_2'\prec t_3'$ are orderly bracketed strings of variables and $t=(t_1(t_2t_3)), t'=((t_1t_2)t_3)$. Then, $t^\mu\neq t'{^\mu}$.

{\it{Proof of Claim}}: Consider first the case where $\check{t}=\check{t}'$. We learn from Proposition \ref{mainprop} that there is a corresponding $\alpha\in\Lambda_4^{t}=\Lambda_4^{t'}$ such that
\begin{equation}
t^{\mu_\alpha^t}=-t'{^{\mu_\alpha^{t'}}}. \label{sgndiff}
\end{equation}
Therefore, we see that for this $\alpha$, $p_\alpha^t$ and $p_\alpha^{t'}$ differ in sign. Consequently, although $\Lambda_4^t=\Lambda_4^{t'}$, the coefficients of $e_4$ in the products $t^\mu$ and $t'{^\mu}$ are different by the uniqueness of the NAF because of the sign difference brought about by (\ref{sgndiff}).

As for the case where $t=(t_1(t_2t_3))$ and $t'=((t_1't_2')t_3')$ have different sets $\{x_{n_1}, \ldots, x_{n_N}\}$ and $\{x_{m_1}, \ldots, x_{m_M}\}$ of variables, we note that the quantities $2^{8n_1+1+\alpha_1}+\cdots+2^{8n_N+1+\alpha_N}$ and $2^{8m_1+1+\beta_1}+\cdots+2^{8m_M+1+\beta_N}$ for $(\alpha_1, \ldots, \alpha_N)\in\{0, \ldots, 7\}^N$ and $(\beta_1, \ldots, \beta_M)\in\{0, \ldots, 7\}^M$ are different by the uniqueness of the binary representation of integers. This in turn implies that $\{p_\alpha^t:\alpha\in\Lambda_0^t\}\neq\{p_\beta^{t'}:\beta\in\Lambda_0^{t'}\}$, which means that the coefficients $\sum_{\alpha\in\Lambda_0^t} p_\alpha^t$ and $\sum_{\beta\in\Lambda_0^{t'}} p_\beta^{t'}$ of $e_0$ in the products $t^\mu, t'^{\mu}$ are different by the uniqueness of the NAF representation of integers. Therefore, $t^\mu\neq t'^{\mu}$. \qed (Claim.)

Finally, if $\langle a_1, a_2, a_3, \ldots\rangle=\vec{a}\leq\vec{b}$, then there exist orderly bracketed strings $t_1\prec t_2\prec t_3$ such that $a_1=t_1{^\mu}, a_2=t_2{^\mu}, a_3=t_3{^\mu}$. We may now conclude that the value of $(a_1(a_2a_3))$ is in $X$ while the value of $((a_1a_2)a_3)$ is in $X^C$, hence $\fr(\vec{a})\cap X\neq\varnothing$ and $\fr(\vec{a})\cap X^C\neq\varnothing$. This concludes the proof that $(\mathbb{O}, \cdot)$ is not a Ramsey algebra.
\end{proof}

\section{Conclusion}
We have seen that finite Moufang loops are Ramsey algebras and so are $M(G, 2)$ for every group $G$. A finite Moufang loop is Ramsey because every given infinite sequence has an element of the loop that occurs infinitely often, which can then be reduced to the identity element through a finite order argument. $M(G, 2)$ is a Ramsey algebra because every sequence of $M$ has a reduction all of whose terms satisfy associativity. For the case of octonion multiplication, our proof showed that there exists a sequence all of whose reductions do not have terms that obey associativity. We will be interested to understand how associativity plays a role in deciding whether a binary system is Ramsey or not through the results contained in this paper. Some work along this lines can be found in \cite{wcTroa}, which investigates a local version of Ramsey algebra.

\section{Acknowledgment}
This work grows out of the first author Zu Yao Teoh's research in the doctoral program at Universiti Sains Malaysia.

The first and third authors gratefully acknowledge the support of the Fundamental Research Grant Scheme No.$\sim$203/PMATHS/6711464 of the Ministry of Education, Malaysia, and Universiti Sains Malaysia. Special thanks also go to the School of Mathematical Sciences' postgraduate allocation fund of Universiti Sains Malaysia as well as the Centre International de Math\'{e}matiques Pures et Appliqu\'{e}es (CIMPA) for their financial supports in making attending the AMC 2016 possible.

\end{document}